\documentclass{gtart_a}
\pdfoutput =1
\usepackage{pinlabel}

\title[Symplectic embeddings of ellipsoids in dimension greater than four]{Symplectic embeddings of ellipsoids\\in dimension greater than four}

\author{Olguta Buse}
\givenname{Olguta}
\surname{Buse}
\address{Department of Mathematics\\Indiana University Purdue University Indianapolis\\\newline
Indianapolis IN 46202\\USA}
\email{buse@math.iupui.edu}
\urladdr{}

\author{Richard Hind}
\givenname{Richard}
\surname{Hind}
\address{Department of Mathematics\\University of Notre Dame\\\newline
Notre Dame IN 46556\\USA}
\email{hind.1@nd.edu}
\urladdr{}

\volumenumber{15}
\issuenumber{4}
\publicationyear{2011}
\papernumber{51}
\startpage{2091}
\endpage{2110}

\keyword{symplectic embedding}
\keyword{packing stability}
\subject{primary}{msc2010}{53D35}
\subject{primary}{msc2010}{57R17}

\received{18 April 2011}
\revised{16 August 2011}
\accepted{13 September 2011}
\published{28 October 2011}
\publishedonline{28 October 2011}
\proposed{Leonid Polterovich}
\seconded{Yasha Eliashberg, Ronald Fintushel}
\editor{Liz}
\version{2}

\arxivreference{}
\arxivpassword{}

\doi{}
\MR{}
\Zbl{}

\AtBeginDocument{\let\bar\wbar}
\def\nb{\nobreak}
\makeatletter\def\makerm#1{\@namedef{#1}{\mathrm{#1}}}\makeatother
\AtBeginDocument{\renewcommand{\r}[1]{\interior{#1}}}
\def\Bbb{\mathbb}
\def\CP{\C`\P}
\let\tsty\textstyle
\def\unfrac#1{#1/}
\def\unpfrac#1#2{#1/(#2)}
\def\upnpfrac#1#2{(#1)/(#2)}
\def\upnfrac#1#2{(#1)/#2}
\def\punfrac#1#2{(#1/#2)}
\def\interior{\mathaccent "7017 }

\newtheorem{theorem}{Theorem}[section]
\newtheorem{proposition}[theorem]{Proposition}
\newtheorem{lemma}[theorem]{Lemma}

\theoremstyle{definition}

\newtheorem{remark}[theorem]{Remark}
\newtheorem{definition}[theorem]{Definition}

\def\N{\mathbb N}
\def\bigmid{\,\bigg|\,}

\newcommand{\CC}{{\mathbb C}}
\newcommand{\NN}{{\mathbb N}}
\newcommand{\RR}{{\mathbb R}}

\AtBeginDocument{
\def\P{\mathbb P}
\renewcommand{\v}[1]{\mbox{\boldmath $#1$}}
}

\begin{document}

\begin{asciiabstract}
We study symplectic embeddings of ellipsoids into balls. In the main
construction, we show that a given embedding of 2m-dimensional
ellipsoids can be suspended to embeddings of ellipsoids in any higher
dimension. In dimension 6,s if the ratio of the areas of any two axes
is sufficiently large then the ellipsoid is flexible in the sense that
it fully fills a ball. We also show that the same property holds in
all dimensions for sufficiently thin ellipsoids E(1,..., a). A
consequence of our study is that in arbitrary dimension a ball can be
fully filled by any sufficiently large number of identical smaller
balls, thus generalizing a result of Biran valid in dimension 4.
\end{asciiabstract}

\begin{abstract} 
We study symplectic embeddings of ellipsoids into balls. In the main
construction, we show that a given embedding of $2m$--dimensional ellipsoids can
be suspended to embeddings of ellipsoids in any higher dimension. In dimension
$6$, if the ratio of the areas of any two axes is sufficiently large then the
ellipsoid is flexible in the sense that it fully fills a ball. We also show
that the same property holds in all dimensions for sufficiently thin ellipsoids
$E(1,\ldots, a)$. A consequence of our study is that in arbitrary dimension a
ball can be fully filled by any sufficiently large number of identical smaller
balls, thus generalizing a result of Biran valid in dimension $4$.
\end{abstract}

\maketitle

\section{Introduction}\label{sec1}

Let $E(a_1, \dots, a_n) \subset \RR^{2n}$ be the ellipsoid \[E(a_1, \dots,
a_n)=\biggl\{\,\sum_{i=1}^n \frac{\pi(x_i^2+ y_i^2)}{a_i} \le 1\biggr\}.\]
Ellipsoids inherit a symplectic structure from the standard form $\omega_0 =
\sum_{i=1}^n dx_i \wedge dy_i$ on $\RR^{2n}$. Then, in our notation, the ball
of capacity $c$ is written \[B^{2n}(c)=E(c, \dots, c).\] Let us also write
$\lambda E(a_1, \dots, a_n)$ and $\lambda B(c)$ for the ellipsoid $E(\lambda
a_1, \dots, \lambda a_n)$ and ball $B(\lambda c)$ respectively. Throughout the
paper the notation \[E(a_1, \dots, a_n) \hookrightarrow E(b_1, \dots, b_n)\]
will mean that for all $\lambda>1$ there exists a symplectic embedding $E(a_1,
\dots, a_n) \hookrightarrow \lambda E(b_1, \dots, b_n)$.

We are interested in the problem of determining when there exists a symplectic
embedding from a given ellipsoid into (an arbitrarily small neighborhood of)
the ball of capacity $c$.

This problem has been completely solved when $n=2$, that is, in dimension $4$,
in the sense that the function \[g(a) := \inf\{\,c\mid E(1,a) \hookrightarrow
B^4(c)\}\] is described by McDuff and Schlenk~\cite[Theorem $1.1.2$]{mcdsch}; see our
\fullref{capac4}.

Here we begin a systematic study of the corresponding functions in higher
dimensions. The main construction that we introduce allows us to extend known
results on embeddings in low dimension to higher dimensional embeddings:

\begin{proposition}
 \label{e}
  Suppose that $$E(a_1, \dots ,a_m) \hookrightarrow E(a'_1, \dots ,a'_m).\vspace*{-6pt}$$
Then also
  $$E(a_1, \dots ,a_m, a_{m+1}, \dots ,a_n) \hookrightarrow E(a'_1, \dots ,a'_m,
a_{m+1}, \dots ,a_n)$$
  for any values $a_{m+1}, \dots ,a_n.$
\end{proposition}

 We will focus especially on dimension $6$, where the problem is to describe
the function of two variables \[f(a,b):=\inf\{\, c\mid E(1,a,b) \hookrightarrow
B^6(c) \}.\]
Note that by symmetry and rescaling we may assume that $1 \le a \le b$.

We are able to describe $f$ completely in particular when $a^2 + b^2 \le 4$ and
also when $a^2 + b^2 \ge 1.41 \times 10^{101}$.
In other words we have optimal embedding results when the ellipsoid is either
relatively close to a ball or in the other extreme when it is, up to scale,
contained in a relatively small neighborhood of a $4$--dimensional ellipsoid.
The results for $a$, $b$ small are contained in the sequence of Lemmas
\ref{emb1}, \ref{emb2}, \ref{emb3}, \ref{emb4}. More known values of $f(a,b)$
are illustrated in \fullref{figure1}. The result for $a$ or $b$ large is
perhaps the main result of our paper and can be stated as follows.

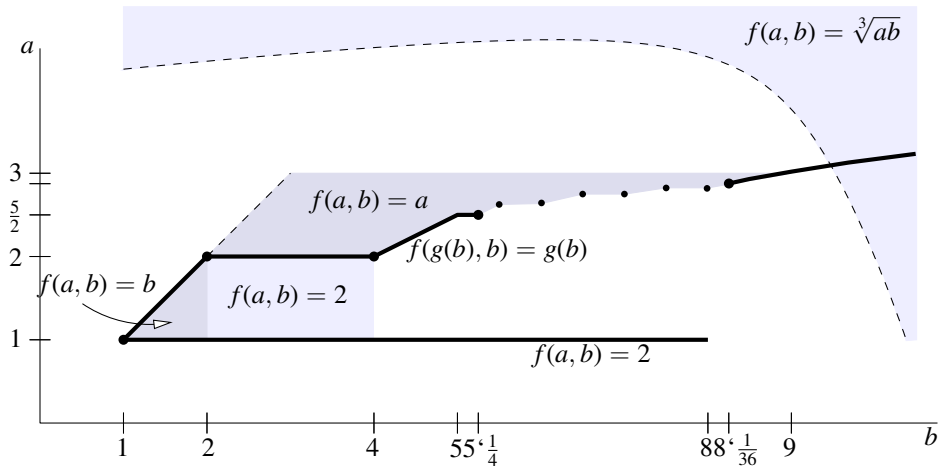
\begin{figure}[ht!]
\labellist
\small\hair2pt
\pinlabel {$1$} [t] at 199 482
\pinlabel {$2$} [t] at 236 482
\pinlabel {$4$} [t] at 307 482
\pinlabel {$5$} [t] at 343 482
\pinlabel {$5 `\smash{\frac14}$} [t] at 354 482
\pinlabel {$8$} [t] at 451 482
\pinlabel {$8 `\smash{\frac1{36}}$} [t] at 464 482
\pinlabel {$9$} [t] at 487 482
\pinlabel {$b$} [t] at 548 487
\pinlabel {$1$} [r] at 158 523
\pinlabel {$2$} [r] at 158 558
\pinlabel {$\frac52$} [r] at 158 576
\pinlabel {$3$} [r] at 158 595
\pinlabel {$a$} [r] at 163 649
\pinlabel {$f(a,b)=b$} [bl] at 162 539
\pinlabel {$f(a,b)=2$} at 271 541
\pinlabel {$f(a,b)=2$} [t] at 402 523
\pinlabel {$f(g(b),b)=g(b)$} [l] at 321 561
\pinlabel {$f(a,b)=a$} at 306 582
\pinlabel {$f(a,b)=\sqrt[3]{ab}$} at 502 656
\endlabellist
\centering
  \includegraphics[width=.95\textwidth]{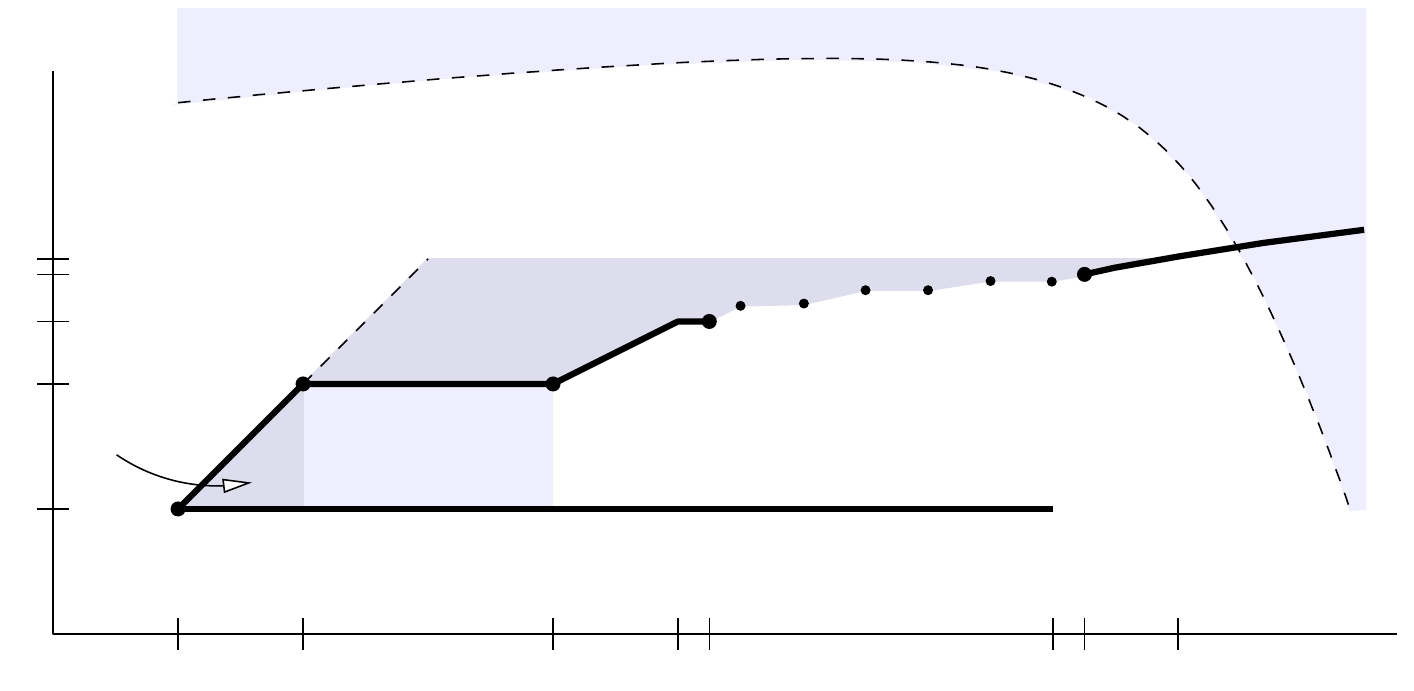}
  \caption{Known values of $f(a,b)$}
   \label{figure1}
\end{figure}

\begin{theorem}
 \label{fullfill}
  If $a^2 + b^2 \ge 1.41 \times 10^{101}$ then $E(1,a,b) \hookrightarrow
B((ab)^{1/3})$.
\end{theorem}

This means that in the given range we have volume filling embeddings, that is,
the only obstruction to embedding this class of ellipsoids into a ball comes
from their volumes. In dimension $4$ the analogous result is that $E(1,a)
\hookrightarrow B( \sqrt{a})$ provided that $a>8\tfrac{1}{36}$.

A consequence of \fullref{fullfill} is a full packing (or filling) result
for higher dimensional balls.
Let $\bigsqcup_k B(c)$ be the disjoint union of $k$ standard
$2n$--dimensional balls of radius ~$r$ and capacity $c=\pi r^{2}$.

\begin{theorem}
 \label{pack}
  For any natural number $ n \geq 3$ there exists a number $M_n$ such that for
all $k \geq M_n$,
  \[\tsty\bigsqcup_k B ( \unfrac{1}{k^{1/n}}  ) \hookrightarrow B^{2n}(1).\]
\end{theorem}

In other words, the round ball can be fully filled by a disjoint union of any
number $k \geq M_n$ of identical balls. For the definition of $M_n$, see
\fullref{def1}.

The
$k$--th packing number of a compact, $`2n`$--dimensional, symplectic
manifold $`(`M``, \omega`)``$~is
\[
  p_k (M,\omega) =
  \frac{\sup_c {\rm Vol}\bigl(\bigsqcup_k B(c)\bigr)}{{\rm Vol}(M,\omega)},
\]
where the supremum is taken over all $c$ for which there exist a
symplectic embedding of $\bigsqcup_k B(c)$ into $(M,\omega)$. Naturally,
$p_k(M,\omega) \leq 1$. When $p_k(M,\omega)=1$ we say that
$(M,\omega)$ admits a full packing by $k$ balls, otherwise we say that
there is a packing obstruction.

Although no general tools are known to compute those invariants for
arbitrary symplectic manifolds, some results can be derived from
complex algebraic geometry using the theory of $J$--holomorphic curves. A first
result that follows from Gromov  \cite{gromov} is that $p_{i}(\CP^n) <1$
for any $1<i < 2^n$. McDuff and
 Polterovich~\cite{MP} computed $p_{i}(\CP^2)$ for $i\leq 9$. They also
proved that $p_{i}(\CP^n) = 1$ whenever $i=k^{n}$ and that
$\lim_{i\rightarrow\infty} p_i(M,\omega) = 1$ for any compact symplectic
manifold. Such results led to the natural question of whether the sequence
$p_i(M,\omega)$ is eventually stable, that is, whether
there is a number $N_{{\rm stab}}(M,\omega)$ such that $p_i(M,\omega)
= 1$ for all $i \geq N_{{\rm stab}}(M,\omega)$. To date, this remains
an interesting open question (see Biran \cite{Bi-SPAG} and Cieliebak et~al~\cite{CHLS} for a
complete discussion). When $M$ is four dimensional, results of McDuff
\cite{McDuffdif} and Lalonde and McDuff \cite{ML} regarding the structure of
symplectic ruled surfaces and introducing inflation techniques then opened the
way to a thorough study of the packing numbers. This study was done by P~Biran
in a sequence of papers \cite{Bi,Bi2} which answered the stability
question positively in the cases of closed symplectic $4$--manifolds whose
symplectic forms (after
rescaling) are in rational cohomology classes. His techniques allowed
him to obtain upper and lower bounds for
$N_{{\rm stab}}(M^{4},\omega)$ which can be
explicitly computed in some cases. In particular, he showed that
$N_{{\rm stab}}(\CP^2) \leq 9$ which, in view of McDuff and
Polterovich's results, is sharp.
Although suspected to be true, until now there were no results in the literature
proving packing stability for a symplectic manifold of dimension larger than
$4$. \fullref{pack} above shows that balls admit full packings by a
sufficiently large number of disjoint identical balls. As the affine part of
$\CP^n$ is a ball, in this language \fullref{pack} gives the following.

\begin{theorem} \label{stability}
Consider $(\CP^n, \omega)$ with the symplectic form induced by the
Fubini--Study metric.
Then $\CP^n $ has packing stability; indeed
\begin{equation}
  p_{i}(\CP^n) =
  1 \quad\text{when } i \geq M_n.
\end{equation}
\end{theorem}

Probably the bound $M_n$ is not optimal, we briefly discuss this in \fullref{improve}.

\subsubsection*{Outline of the paper}

In \fullref{sec2} we describe our basic embedding construction in
\fullref{e} and apply it to deduce various values of $f(a,b)$, in
particular for~$a$,~$b$ sufficiently small. We remark that E~Opshtein
\cite{opshtein} has also given a construction for embedding ellipsoids. The
focus of his work is perhaps embeddings into closed manifolds, but there is
still some overlap with our own results. We mention this also in ~\fullref{sec2}.

\fullref{sec3} gives the proof of \fullref{fullfill} and also constructs some
volume filling embeddings of balls by ellipsoids in any dimension.

\fullref{pack} is proven in \fullref{sec4}, by combining the volume filling
results of \fullref{sec3} with a construction from toric geometry.

\subsubsection*{Acknowledgements} The second author would like to thank Kaoru Ono for an
enlightening discussion on toric decompositions, and both authors thank Dusa
McDuff and an anonymous referee for helpful comments on the text.

\section{Embedding ellipsoids}\label{sec2}

\subsection{The construction}

Here we give the proof of our basic embedding construction. Before giving a
formal proof of \fullref{e} we outline the general idea. Identifying
each $\Bbb R^{2n}$ with~ $\Bbb C^n$, let $H_s\co  \CC^m \times [0,1] \to \RR$ be a
$1$--parameter family of possibly time-dependent Hamiltonian functions on
$\CC^m$ and set $H(z,s,t)=H_s(z,t)$, so $H\co \CC^m \times \RR \times [0,1] \to
\RR$. Then for any function $f\co  \CC^{n-m} \to \RR$, the function $H(z,f,t)$ can
be thought of as a Hamiltonian ~$G$ on $\CC^n$. Let $\pi\co \CC^n \to \CC^{n-m}$ be
the projection on the last $n-m$ complex coordinates, then $F=f \circ \pi\co\CC^n
\to \RR$ is an integral of the motion of $G$. Indeed, at all times $t$, we have
$\{G,F\} = dG(X_F) = \frac{\partial H}{\partial s} df(X_f) =0$. Therefore the
flow of $G$ preserves the level sets of $F$, and restricted to a level
$\Sigma_c = \{F=c\}=\CC^m \times \{f=c\}$ the first~ $m$ (complex) components of
the flow are exactly those of the Hamiltonian~ $H_c$. In other words, let $\Phi$
be the time~$1$ flow of $G$ and $\phi_c$ be the time~$1$ flow of $H_c$. Then if
$(z,w) \in \CC^m \times \CC^{n-m} = \CC^n$ has $F(z,w)=f(w)=c$, we have
$\Phi(z,w)=(\phi_c(z),w')$ where $f(w')=c$. Thus a domain $D \subset \CC^n$
whose fibers $\pi^{-1}(w) \cap D = D_w = D_{f(w)}$ depend only on $f(w)$ will
be mapped under $\Phi$ to a domain $D'$ with fibers $\pi^{-1}(w) \cap D' =
\phi_{f(w)}(D_{f(w)})$.

Returning to \fullref{e}, and still
identifying $\Bbb R^{2n}$ with $\Bbb C^n$, we can write
\[
  E(a_1, \dots, a_n) =
  \biggl\{(z_1, \dots ,z_n) \in \CC^n \bigmid \frac{\pi|z_1|^2}{a_1} + \cdots +
\frac{\pi|z_n|^2}{a_n} \leq 1 \biggr\}.
\]
Using the notation above, the fibers of $E(a_1, \dots, a_n)$ over $\CC^{n-m}$ are
ellipsoids $rE(a_1, \dots, a_m)$ where $r=1 - \unfrac{\pi |z_{m+1}|^2}{a_{m+1}} - \cdots - \unfrac{\pi|z_n|^2}{a_n}$. Roughly speaking, we will apply our general
idea to a $1$--parameter family of Hamiltonian functions $H_r$ whose
corresponding flows map $rE(a_1, \dots, a_m)$ into $rE(a'_1, \dots, a'_m)$.

\begin{proof}[Proof of \fullref{e}] Fixing a $\lambda >1$, it is required to
show that there exists a symplectic embedding $E( a_1, \dots ,a_m, a_{m+1},
\dots, a_n  ) \to \lambda E(a'_1, \dots ,a'_m, a_{m+1}, \dots, a_n)$.

By hypothesis, we have $E(a_1, \dots ,a_m) \hookrightarrow E(a'_1, \dots
,a'_m)$. By the Extension after Restriction Principle (see Schlenk~\cite[page 7]{Sh1})
our hypothesis implies that there exists a $\mu >1$ and a Hamiltonian
diffeomorphism mapping $\mu E(a_1, \dots ,a_m) \to \lambda E(a'_1, \dots
,a'_m)$. Suppose that this Hamiltonian diffeomorphism is the time~$1$ flow
corresponding to a Hamiltonian function $H\co  \Bbb C^m \times [0,1] \to \Bbb R$.

Observe that for any $r>0$ the Hamiltonian $H_r$ defined by $H_r(z,t)=r
H(\unfrac{z}{\sqrt{r}},t)$ for $z \in \Bbb C^m$ generates a flow with time 1 map
taking $r \mu E(a_1, \ldots, a_m) \to r\lambda E(a'_1, \ldots, a'_m)$. (For
this, recall that in our notation the map $(z_1, \ldots ,z_m) \mapsto
(\sqrt{r}z_1, \ldots ,\sqrt{r}z_m)$ takes $E(a_1, \ldots ,a_m)$ onto $rE(a_1,
\ldots ,a_m)$.)

For $z \in \Bbb C^n$, let
\begin{gather*}
  r(z)=r(z_{m+1}, \ldots, z_n) =
  1 - \frac{\pi |z_{m+1}|^2}{\mu a_{m+1}} - \cdots - \frac{\pi|z_n|^2}{\mu
a_n}.
\\[1ex]\tag*{\hbox{\rlap{Define}}}
K(z_1,\dots ,z_n,t)=H_{r(z_{m+1},\dots ,z_n)}(z_1,\dots ,z_m,t)
\end{gather*}
for $z$ with $\unfrac{\pi |z_{m+1}|^2}{a_{m+1}} +\nb \cdots +\nb \unfrac{\pi|z_n|^2}{a_n} \le
1$ and extend the function arbitrarily to the remainder of $\Bbb C^n \times
[0,1]$. Note that $\unfrac{\pi |z_{m+1}|^2}{a_{m+1}} + \cdots +
\unfrac{\pi|z_n|^2}{a_n} \le 1$ implies that $r(z) \ge 1 - \unfrac{1}{\mu} >0$ and
so $K$ is well defined. We claim that $K$ generates a Hamiltonian
diffeomorphism~ $\phi$ as required.

More precisely, we make the following claim. Suppose that $z \in \Bbb C^n$ is such
that $\unfrac{\pi |z_{m+1}|^2}{a_{m+1}} + \cdots + \unfrac{\pi|z_n|^2}{a_n} =k \le
1$ and $\phi(z)=w=(w_1, \dots ,w_n)$. Then
\begin{enumerate}
\item $\unfrac{\pi |w_{m+1}|^2}{a_{m+1}} + \cdots + \unfrac{\pi|w_n|^2}{a_n} =k$;

\item $\unfrac{\pi |w_1|^2}{a'_1} + \cdots + \unfrac{\pi|w_m|^2}{a'_m} \le
\lambda(1-k)$. 
\end{enumerate}
Given this, we have $\unfrac{\pi |w_1|^2}{a'_1} + \cdots + \unfrac{\pi|w_n|^2}{a_n} \le
\lambda(1-k) +k \le \lambda$ and so $\phi(z)=w \in \lambda E(a'_1, \dots ,a'_m,
a_{m+1}, \dots, a_n)$ as required.

Statement (1) follows because on the region $\{\unfrac{\pi
|z_{m+1}|^2}{a_{m+1}} + \cdots + \unfrac{\pi|z_n|^2}{a_n} \le 1\}$ the
Hamiltonian flow of $K$ preserves $r$ and hence $\smash{\unfrac{\pi
|z_{m+1}|^2}{a_{m+1}} + \cdots + \unfrac{\pi|z_n|^2}{a_n}}$.

For statement (2) note that if $\unfrac{\pi |z_{m+1}|^2}{a_{m+1}} + \cdots +
\unfrac{\pi|z_n|^2}{a_n} =k$ then the partial derivatives of our Hamiltonian in
the $z_1, \dots ,z_m$ directions are equal to the corresponding derivatives of
$H_{1-\unfrac{k}{\mu}}(z_1, \dots ,z_m)$.

Now, $(z_1, \dots ,z_m) \in (1-k)E(a_1, \dots ,a_m) \subset
(1-\unfrac{k}{\mu})\mu E(a_1, \dots ,a_m)$. Thus as the flow of
$H_{1-\unfrac{k}{\mu}}$ takes $(1-\unfrac{k}{\mu}) \mu E(a_1, \ldots, a_m) \to
(1-\unfrac{k}{\mu}) \lambda E(a'_1, \ldots, a'_m)$ at time~ $1$ we have $(w_1,
\dots w_m) \in (1-\unfrac{k}{\mu}) \lambda E(a'_1, \ldots, a'_m)$. In other
words, $w$ satisfies 
$\smash{\unfrac{\pi |w_1|^2}{a'_1} + \cdots + \unfrac{\pi|w_m|^2}{a'_m} \le
(1-\unfrac{k}{\mu}) \lambda} < \lambda(1-k)$ as claimed.
\end{proof}

\begin{remark} As mentioned in the introduction there is related work of  
Opshtein which we outline here, for convenience focusing on the case of
embeddings into $\CP^3(c)$, complex projective space equipped with the
Fubini--Study form scaled such that lines have symplectic area $c$. The
symplectic manifold $\CP^3(c)$ is of special interest to us as the affine
part is symplectomorphic to the $6$--ball of capacity $c$. Opshtein observes the
following.

\begin{theorem}[Opshtein \cite{opshtein}]
  Let $\Sigma \subset \CP^3(c)$ be a smooth holomorphic hypersurface of
degree $k$. Then any symplectic embedding $E(a,b) \hookrightarrow \Sigma$
extends to a symplectic embedding $E(\unfrac{c}{k},a,b) \hookrightarrow
\CP^3(c)$.
\end{theorem}

For example, if $k=1$ then $\Sigma$ is a copy of $\CP^2$ which contains an
embedded ball~ $B^4(c)$ of capacity $c$. Thus if $E(a,b) \hookrightarrow B(c)$
then we also find an embedding $E(c,a,b) \hookrightarrow \CP^3(c)$. Under
the same hypotheses our \fullref{e} also gives an embedding $E(c,a,b)
\hookrightarrow B^6(c) \subset \CP^3(c)$. We expect the two embeddings are
symplectically isotopic.
\end{remark}

\subsection{Embedding obstructions}

To check that our constructions are sharp we rely only on the volume
obstruction and on the Ekeland--Hofer capacities~\cite{ekho1,%
ekho2}. The volume obstruction says the following.

\begin{proposition}[Liouville's Theorem]
 \label{ob1}
  If $E(a_1, \dots, a_n) \hookrightarrow E(b_1, \dots, b_n)$ then $a_1 \cdots
a_n \le b_1 \cdots b_n$.
\end{proposition}

The Ekeland--Hofer capacities give an infinite sequence of numbers $c_k(E(a_1,
\dots, a_n))$ associated to an ellipsoid. In our situation we can take the
definition to be as follows.

\begin{definition} $c_k(E(a_1, \ldots, a_n))$ is the $k$--th number in the
ordered sequence (with repetitions if necessary) of numbers in the set
$\{k_1a_1, \ldots ,k_na_n \mid k_i \in \NN\}$.
\end{definition}

\begin{theorem}{\rm \cite{ekho1,ekho2}}\qua
 \label{ob2}
  If $E(a_1, \dots, a_n) \hookrightarrow E(b_1, \dots, b_n)$ then $c_k(E(a_1,
\dots, a_n)) \le c_k(E(b_1,
  \dots, b_n))$ for all $k$.
\end{theorem}

Note that $c_k(\lambda E(a_1, \dots, a_n)) = \lambda c_k(E(a_1, \dots, a_n))$.
Therefore if $E(a_1, \dots, a_n) \hookrightarrow E(b_1, \dots, b_n)$ and
$c_k(E(a_1, \dots, a_n)) = c_k(E(b_1, \dots, b_n))$ for some $k$, we know that
the embedding is optimal in the sense that there is no embedding $E(a_1, \dots,
a_n) \hookrightarrow \mu E(b_1, \dots, b_n)$ for any $\mu <1$.

\subsection{Some calculations in dimension $3$}

Here we give some optimal embeddings for ellipsoids $E(1,a,b)$ when $a$ and $b$
are relatively small. We recall the definitions of the functions $f$ and $g$
from the introduction.
$$\eqalign{
  f(a,b) &:= \inf \{\, c \mid E(1,a,b) \hookrightarrow B^6(c) \}, \cr
  g(a)   &:= \inf \{\, c \mid E(1,a) \hookrightarrow B^4(c)  \}.
}$$
The function $g$ is completely determined by McDuff and Schlenk~\cite{mcdsch}. In this
paper we apply only a small amount of information about $g$ which is summarized
in the following proposition.

\begin{proposition}[McDuff--Schlenk {{\cite[Theorem 1.1.2]{mcdsch}}}]
 \label{capac4}
  If $b = 4$ or $b > \smash{8\tfrac{1}{36}}$, then $g(b) = \sqrt{b}$.
\end{proposition}

We always assume without loss of generality that $1 \le a \le b$.

\begin{lemma}
 \label{emb1}
  $f(g(b),b) = g(b)$.
\end{lemma}

\begin{proof}
We observe that
\[E(1,g(b),b) \cong E(g(b),1,b) \hookrightarrow E(g(b), g(b), g(b)) = B(g(b)),
\]
where the arrow follows from \fullref{e} and the definition of $g$.
Therefore $f(g(b),b) \le g(b)`$.

Now, it follows from \fullref{capac4} that $g(b)= \sqrt{b}$ whenever $b
\ge 9$. Thus the embedding is optimal in this case by \fullref{ob1}.

If $b \le 9$ then $g(b) \le 3$. In the first case suppose that $1 \le g(b) \le
2$. Then \[c_2(E(1,g(b),b))=g(b)=c_2(B(g(b))).\] In the second case, if $2 \le
g(b) \le 3$ then \[c_3(E(1,g(b),b))=g(b)=c_3(B(g(b))).\] Thus the embedding is
also optimal in both of these cases by \fullref{ob2}.
\end{proof}

\begin{lemma} \label{emb2} Suppose that $a \le 3$ and $g(b) \le a$. Then
$f(a,b)=a$. \end{lemma}

\begin{proof} The embedding construction here is as in the previous proof of
\fullref{emb1}, and our hypothesis are such that the third Ekeland--Hofer
capacity again implies that it is optimal.
\end{proof}

\begin{lemma} \label{emb3} Suppose that $1 \le a \le b \le 2$. Then $f(a,b)=b$.
\end{lemma}

\begin{proof}
$E(1,a,b) \hookrightarrow B(b)$ simply by inclusion. But
\[c_3(E(1,a,b))=b=c_3(B(b))\] for our range of $a$, $b$, so the inclusion is
optimal.
\end{proof}

\begin{remark}
It is true in any dimension that if $a_n \le 2a_1$ then $E(a_1, \dots ,a_n)
\hookrightarrow B(c)$ if and only if $c \ge a_n$. This was established in the
case of $n=2$ by Floer, Hofer and Wysocki in \cite{fhwI} as an application of
symplectic homology. The theorem stated here is \cite[Theorem $1$]{Sh1},
where Schlenk gives a simple proof by applying the $n$--th Ekeland--Hofer capacity
as in our proof of \fullref{emb3} above. We thank the referee for reminding
us of this slightly strange history.
\end{remark}

\begin{lemma}
 \label{emb4}
  Suppose that $1 \le a \le 2 \le b \le 4$. Then $f(a,b)=2$.
\end{lemma}

\begin{proof}We have the embeddings \[E(1,a,b) \cong E(a,1,b) \hookrightarrow E(a,2,2)
\hookrightarrow B(2)\] where the first arrow follows from \fullref{e}
since $g(b)=2$ (see \cite[Figure~ 1.1]{mcdsch}) and the second arrow is the
inclusion.

The third Ekeland--Hofer capacity
\[c_3(E(1,a,b))=2=c_3(B(2))\]
and so our construction is optimal.
\end{proof}

We give one final computation which will rely on the following.

\begin{lemma}
 \label{olga0}
  $E(1,1,8) \hookrightarrow B(2)$.
\end{lemma}

\begin{proof}
This is a particular case of \fullref{olga3} below.
\end{proof}

\begin{lemma}
 \label{emb5}
  Suppose that $2 \le b \le 8$. Then $f(1,b)=2$.
\end{lemma}

\begin{proof} $f(1,b)$ is an increasing function of $b$.  By \fullref{emb4}, we know that
$f(1,4)=\nb2$  and \fullref{olga0} says that $f(1,8)=2$.
Thus $f(1,b)$ is in fact constant on the interval $4 \le b \le 8$.
\end{proof}

\section{Volume filling embeddings}\label{sec3}

In this section we prove \fullref{fullfill2}, a more precise version of
\fullref{fullfill}. First, in \fullref{sec31} we recall two theorems
of D~McDuff on embeddings in dimension $4$. The first reduces an ellipsoid
embedding problem to one of embedding a disjoint union of balls into a ball,
the second gives necessary conditions for embedding a disjoint union of balls.
We close the subsection with two useful consequences. In \fullref{sec32}
we use these theorems to derive some preliminary results on embedding
ellipsoids in higher dimensions. We think these are quite interesting in
themselves; the main result is \fullref{olga4} which will be applied in
\fullref{sec4} to give our full packing result. Finally in \fullref{sec33}
we prove \fullref{fullfill2}.

\subsection{Four dimensional embeddings} \label{sec31}

Here we review some results of McDuff which allow one to translate a
$4$--dimensional ellipsoid embedding problem into one of a disjoint union of
balls, and then give an algebraic solution for the ball embedding problem.

\begin{proposition}[McDuff \cite{mcduff1,mcduff2}]
 \label{mcdprop}
  Let $e,f,c,d$ be positive integers with $e \leq f$ and $c \leq d$. There
exists a weight expansion
  $W(e,f)$ associated to any pair of integers such that a symplectic  embedding
of ellipsoids
  $E(e,f) \longrightarrow E(c,d)$ exists if and only if there exists a
symplectic embedding of balls
  \begin{equation}
   \label{mcduffballellipse}
    \bigl(\tsty\bigsqcup B(W(e,f)) \bigr) \cup \bigl( \bigsqcup B(W(d-c,d)) \bigr)
\hookrightarrow B(d).
  \end{equation}
\end{proposition}

Let us explain the weight sequences $W(e,f)$, as they are defined for instance
in \cite{mcduff2}.
\begin{equation}\label{w}
 W(e,f) =( X_0^{\times l_0},  X_1^{\times l_1}, \ldots,  X_K^{\times l_K}),
 \end{equation}
 where the multiplicities $l_i$ are the entries in the continued fraction
expansion
\[
  \frac{f}{e} =
  [l_0;l_1,\ldots,l_k] = l_0 + \cfrac{1}{l_1+\cfrac{1}{\ddots_{\ \tsty\unfrac{1}{l_k}}}}
\]
and the entries $X_i$ are defined inductively as follows:
 \begin{equation}\label{x}
 X_{-1} =f, X_0= e,X_{i+1}= X_{i-1}- l_i X_i.
  \end{equation}
We use the notation
\begin{equation}
 \label{b}
\tsty  \bigsqcup B(W(e,f)) :=
  \bigl( \bigsqcup_{i=1}^{l_0} B(X_0) \bigr) \cup
   \bigl( \bigsqcup_{i=1}^{l_1} B(X_1) \bigr) \cup \cdots \cup
   \bigl( \bigsqcup_{i=1}^{l_K} B(X_K) \bigr).
\end{equation}
McDuff and Polterovich~\cite{MP} equated the problem of embedding
$M$ disjoint balls with understanding the symplectic cone of the $M$--fold
blow-up
$X_{M}$ of $\CP^{2}$.
The structure of the cone was intensely studied by Biran \cite{Bi} and later
T-J~Li and Liu \cite{li-liu} and B-H~Li and T-J~Li \cite{li-li}. To explain their results, denote
by $L, E_1, \ldots, E_M \in H_2(X_m, \Z)$ the homology classes of the line and
the exceptional divisors, by $-K:=3L-\sum_{i} E_{i}$ the anticanonical class,
and by $l, e_1,\ldots, e_M$ their Poincar{\'e} duals. With respect to the basis
$L, -E_1, \dots ,-E_M$ we represent homology classes by $(M{+}1)$--tuples  $(d,
\bar{m})$. We fix a symplectic form $\omega_M$ on $X_M$  obtained by
symplectically blowing up $(\CP^{2},\omega)$ with the standard Fubini--Study
form $\omega$.
Define $\mathcal{C}_M$ to be the set of all cohomology classes in $H^2(X_M)$
that can be represented by symplectic forms whose first Chern classes are
Poincar{\'e} dual to $-K$.
Next define the exceptional cone to be the set of
homology classes $\mathcal{E}_M \subset H_2(X_M, \Z)$,
\[
  \mathcal{E}_{M} :=
  \bigl\{E \mid E\cdot E=-1,\ \text{$E$ is represented by an embedded
              $\omega_M$--symplectic sphere} \bigr\}.
\]
Since the classes $E$ in the exceptional cone have  nontrivial Gromov
invariants, the definition of the exceptional cone is independent of the choice
of $\omega_M$.
 B-H~Li and T-J~Li~\cite{li-li} showed that 
\[
  \mathcal{C}_{M} =
   \bigl\{\,\alpha \in H^2(X_M) \mid\alpha^2>0 \text{~and $\alpha(E) >0$ for all }
E\in \mathcal{E}_M  t\,\bigr\}.
\]
Given these definitions, the work of McDuff and Polterovich, Biran, T-J~Li and Liu and
B-H~Li and T-J~Li, gives the following criteria for embedding disjoint unions of balls.

\begin{proposition}{\rm{\cite{MP,Bi,li-li,Li-Li-Diff,li-liu}}}\qua
 \label{mcdprop2}
A symplectic ball embedding
$$\tsty\bigsqcup_{ i = 1, \ldots, M} B^4(w_i) \hookrightarrow B^{4}( \mu)$$ exists if, and
only if, $ \mu l - \sum_{i=1}^M w_i e_i \in \mathcal{C}_M$. This is equivalent
to the following two conditions.

\begin{enumerate}

\item $ \mu^2 \ge \sum_{i=1}^M w^2_i $.

\item For any $(d',\bar{m}) \in \mathcal{E}_M$, we have
\begin{equation}\label{dineq}
d'  \mu \ge \sum_{i=1}^M m_i w_i.
\end{equation}
\end{enumerate}
\end{proposition}

We close this subsection with the following two applications of these results.

\begin{proposition}
 \label{olga2}
  For any $k,x \in \NN$, the following embedding holds:
  \begin{equation}
    E(1,k^{2x+1}) \hookrightarrow E(k^x,k^{x+1}).
  \end{equation}
\end{proposition}

\begin{proof}
We apply \fullref{mcdprop} with $e=1, f=k^{2x+1}, c=k^x, d=k^{x+1}$.
Then the equivalent embedding \eqref{mcduffballellipse} (see the notation
\eqref{b}) becomes
\begin{equation}
 \label{embball1}
\tsty  \bigsqcup B\big( W(1,k^{2x+1}) \big) \cup
  \bigsqcup B\big( W(k^{x+1}-k^x,k^{x+1}) \big) \hookrightarrow B(k^{x+1}).
\end{equation}
The continuous fraction expansions are $\unfrac{1}{\mskip-0.5mu k^{2x+`1}}``=[0;k^{2x+`1}]$
and  $\upnfrac     {k^{x+`1}`-\nb k^x}{\mskip-0.5mu k^{x+`1}}``=[0;1,k-1]$, so the vectors described in
\eqref{x} express this last embedding as
\begin{equation}
 \label{embball2}
\tsty  \bigl( \bigsqcup_{i=1}^{k^{2x+1}} B_i(1) \bigr) \cup
   B(k^{x+1}-k^x) \cup
   \bigl( \bigsqcup_{j=1}^{k-1} B(k^x) \bigr) \hookrightarrow B(k^{x+1}).
\end{equation}
This embedding can be seen from a toric viewpoint. Indeed, there exists a toric
decomposition of the ball of capacity $k^{x+1}$ containing an open ball of
capacity $k^{x+1} -\nb k^x$ and the preimages of $2k-1$ polytopes of capacity
$k^x$ as shown in \fullref{figure2}. For the fact that these open toric
manifolds admit embeddings of an open ball of the stated capacity, see
Traynor~\cite[Proposition $5.2$]{traynor}. But a ball of capacity $k^x$ can be filled
with $k^{2x}$ balls of capacity $1$; see \cite[Construction~$3.2$]{traynor} 
for an explicit construction. If we decompose $k$ of our $2k-1$ such
balls in this way then we get the embedding as required.

\begin{figure}[ht!]
\labellist
\small\hair2pt
\pinlabel {$k^x `=` 25$} at 190 591
\pinlabel {$k^{x+1}-k^x`=`100$} at 226 390
\endlabellist
\centering
  \includegraphics[scale=.65]{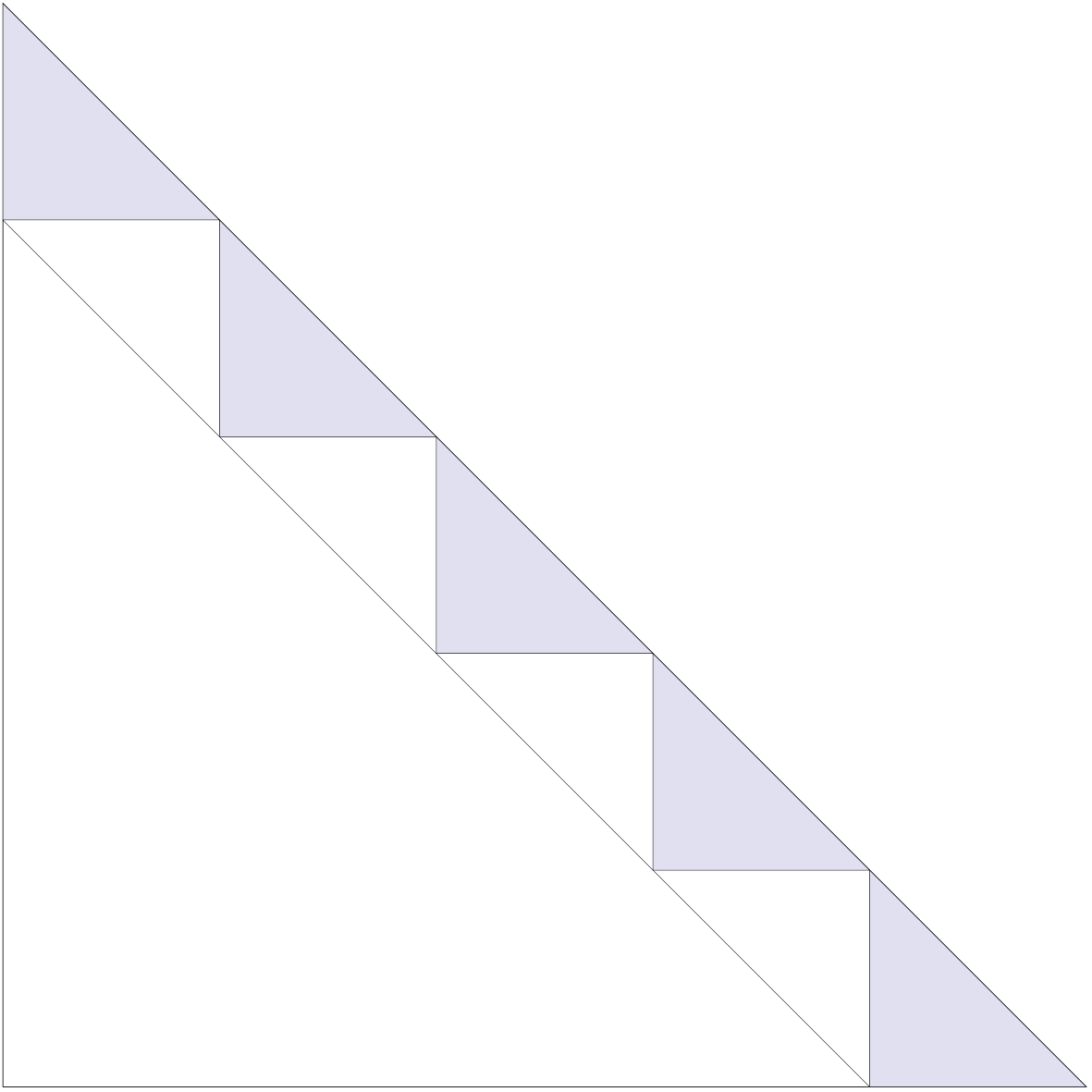}
  \caption{A toric decomposition of the ball when $k=5,\ x=2$. The same toric
packing strategy applies to
       any $k,x $ natural.}
   \label{figure2}
\end{figure}

This concludes the proof of \fullref{olga2}.
\end{proof}

\begin{lemma}\label{lambdatrick}
If $ \sqrt{\unfrac {2}{3}} \leq \lambda \leq 1$ and $b \geq 9$, then $E(1,b)
\hookrightarrow   E(\lambda \sqrt{b}, \lambda^{-1} \sqrt{b}) $.
\end{lemma}

\begin{proof}
The proof largely follows the method used by McDuff and Schlenk in
\cite[Corollary $1.2.4$]{mcdsch} where they establish the case $\lambda=1$.
It is sufficient to consider the case when both $\lambda= \unfrac{u}{v}$ and
$\sqrt{b} = \unfrac{p}{q}$ rational, with $p \geq 3q$, $\sqrt{\unfrac{2}{3}} \leq
\unfrac{u}{v} \leq 1$.

Then we need to prove that
\begin{equation}\label{lambdaemb}
E\biggl( 1, \frac{p^2}{q^2} \biggr) \hookrightarrow E\biggl( \frac{up}{vq},
\frac{vp}{uq} \biggr)
\end{equation}
and the latter is equivalent to showing that there exists an embedding
\begin{equation}\label{lambdaemb2}
E(uvp^2,uvq^2) \hookrightarrow E( u^2pq, v^2pq).
\end{equation}
Using now \fullref{mcdprop} we see that this is equivalent to the
existence of the following ball embedding:
\begin{equation}
 \label{uvballs}
 \tsty \bigl( \bigsqcup B\big( W(uvq^2,uvp^2) \big) \bigr) \cup
   \bigl( \bigsqcup B\big( W((v^2-u^2)pq,v^2pq) \big) \bigr) \hookrightarrow B(
v^2pq  ).
\end{equation}
We will use the criteria from \fullref{mcdprop2} to show that the
embedding \eqref{uvballs} does indeed exist.

The first condition is clearly satisfied because our embeddings are volume
preserving.
To  verify the second, note that since $(d',\bar{m}) \in \mathcal{E}_M$ we
have $3 d'-1 = \smash{\sum_{i=1}^M m_i} $.
\begin{align*}\tag*{\hbox{\rlap{Let}}}
\tsty\bigsqcup B\big( W(uvq^2,uvp^2) \big)&\tsty= \bigsqcup_{i=1}^{M_1}  B(w_i),\\
\tag*{\hbox{\rlap{and}}}\tsty\bigsqcup B\big( W((v^2-u^2)pq,v^2pq) \big) &\tsty= \bigsqcup_{i=M_1+1}^{M}  B(w'_i).
\end{align*}
We have
\begin{equation}
  \sum_{i=1}^{M_1} m_i w_i + \sum_{i=M_1+1}^{M} m_i w'_i \leq 
  \sum_{i=1}^{M_1} uv q^2 m_i + \sum_{i=M_1+1}^{M} \big( (v^2-u^2)pq \big) m_i
\end{equation}
since each $w_i \le uvq^2$ and each $w'_i \le (v^2-u^2)pq$. Therefore
\begin{align}
 \label{thegrossineq}
  \sum_{i=1}^{M_1} m_i w_i + \sum_{i=M_1+1}^{M} m_i w'_i &\leq
  \max \big\{ uvq^2, (v^2-u^2)pq) \big\} \sum_{i=1}^{M}  m_i  \\
&=  \max \big\{ uvq^2, (v^2-u^2)pq) \big\} (3 d'-1).
\notag\end{align}
Thus, to show the inequality
 \eqref{dineq}, it is sufficient to verify that  both of the following
inequalities hold:
\begin{align}
uvq^2 (3 d'-1) &\leq v^2 pq d',
\\
(v^2-u^2)pq (3 d'-1) &\leq v^2 pq d'.
\end{align}
The first inequality is guaranteed if $3uvq^2 \leq  v^2 pq $, or $\unfrac{u}{v}
\leq \punfrac{1}{3} \punfrac{p}{q}$. This is certainly true as $\lambda \leq 1$ and
$\sqrt{b} \geq 3$.

The second is guaranteed if $3(v^2-u^2)pq \leq v^ 2 pq$, or, $\unfrac{u^2}{v^2}
\geq \unfrac{2}{3}$. This is true as well since we assumed $\lambda^2 \geq
\unfrac{2}{3}$.
\end{proof}

\subsection{Applications to higher dimensions} \label{sec32}

We start this section by applying \fullref{olga2} to obtain embeddings
in higher dimensions.

For brevity, we will use the notation  $E(a_1, a_2,\ldots ,a_i^{\times m},
a_{i+m}, \ldots a_n)$ if an entry $a_i$ is repeated $m$ times.

\begin{lemma}
 \label{olga3}
  For any $k,n \in \NN$,
  \begin{equation}\label{kpowers}
    E(1^{\times (n-1)},k^n) \hookrightarrow B(k).
  \end{equation}
\end{lemma}

\begin{proof}
We fix $k$ and proceed by induction on $n$. First, \eqref{kpowers} is obviously
true when $n = 1$, while the case  $n=2$, first proved in \cite{mcduff1}, can
be easily seen as part of  McDuff and Schlenk's results on the values of $g(b)$
from \fullref{capac4}.

Let $n \geq 3$. If $n$ is odd, say $n = 2m+1$, we have
\[
  E(1^{\times (n-1)}, k^n) \hookrightarrow
  E(1^{\times (n-2)}, k^m,k^{m+1}) =
  E(1^{\times (m-1)}, k^m, 1^{\times m}, k^{m+1}),
\]
where the first embedding is from \fullref{olga2}.
By the induction hypothesis and \fullref{e}, the final ellipsoid embeds
into
\[
  E(k^{\times m}, k^{\times (m+1)}) =
  E(k^{\times n}).
\]
On the other hand, if $n$ is even, say $n = 2m$, then
\[
  E(1^{\times (n-1)}, k^n) =
  E(1^{\times m}, 1^{\times (m-1)}, ( k^2)^m),
\]
which by the induction hypothesis (with $k^2,m$ instead of $k,n$) and
\fullref{e} again embeds into
\[
  E(1^{\times m}, ( k^2)^{\times m}).
\]
But by using  \fullref{e} repeatedly and the fact that $E(1,k^2)
\hookrightarrow E(k,k)$ (see again \fullref{capac4}) we can split every
instance of $k^2$ into two copies of $k$, which ends the proof.
\end{proof}

Next, \fullref{lambdatrick} will be used in order to prove a similar result
to \fullref{olga3}, but replacing $k^n$ with a real number.

\begin{definition} \label{def1} Define the sequence $M_i$, $i \ge 2$
inductively as follows:
$M_2 = 8 \tfrac{1}{36}$, $M_n = \max (M_{n-1}^2, \beta_n)$, where
\begin{equation}
  \beta_n =
  \biggl( \frac{\sqrt[2n-2]{3}}{\sqrt[2n-2]{3}- \sqrt[2n-2]{2}} \biggr )^{\unpfrac
{2n(n-1)}{n-2}}
\end{equation}
\end{definition}

 We are now in a position to state the main result of this subsection.

\begin{proposition} \label{olga4} If $b$ is any real number with 
$b
\geq M_n$, then $E(1^{\times (n-1)},b) \hookrightarrow\nb   B( b^{\unfrac{1}{n}}) $.
\vadjust{\goodbreak}
\end{proposition}

\begin{proof}
We will proceed by induction with respect to the dimension $n$. The result
holds when $n=2$ by \cite{mcdsch}; see \fullref{capac4}.
Assume that the result holds for ellipsoids of complex dimension less or equal
than $n-1$.
\[
  \lambda =
  \biggl( \frac{\lfloor b^{\upnpfrac{n-2}{2n(n-1)}}
\rfloor}{b^{\upnpfrac{n-2}{2n(n-1)}}} \biggr)^{n-1}.
\leqno{\hbox{We set}}\]
Then since $b \ge \beta_n$ we observe that
\[
  \sqrt{\frac{2}{3}} \leq
  \biggl( \frac{ b^{\upnpfrac{n-2}{2n(n-1)}}-1}{b^{\upnpfrac{n-2}{2n(n-1)}}}
\biggr)^{n-1} \leq
  \lambda \le
  1.
\]
By \fullref{lambdatrick} and \fullref{e} we have the embedding
\begin{equation}
  E(1^{\times (n-1)},b)
    \hookrightarrow E(1^{\times (n-2)}, \lambda^{-1} \sqrt{b} , \lambda
\sqrt{b} ).
\end{equation}
Meanwhile, since $\lambda^{-1} \sqrt{b} \geq
  M_{n-1}$, by the induction hypothesis we have the embedding
\vspace*{-3pt}
\begin{multline}
  E ( 1^{\times (n-2)}, \lambda^{-1} \sqrt{b} , \lambda \sqrt{b}  )
\hookrightarrow
    E \bigl(  \bigl( (\lambda^{-1} \sqrt{b})^{\unpfrac{1}{n-1}} \bigr)^{\times(n-1)},
\lambda \sqrt{b}  \bigr)  \\
=   ( \lambda^{-1}\sqrt{b}  )^{\unpfrac{1}{n-1}} E ( 1^{\times(n-1)}, z
 ),
\end{multline}
\[
  z =
  \frac{\lambda \sqrt{b}}{(\lambda^{-1} \sqrt{b})^{\unpfrac{1}{n-1}}} =
  ( \lambda^{\unpfrac{n}{n-1}}  ) \cdot  ( b^{\upnpfrac{n-2}{2n -2}}  ) =
   \lfloor b^{\upnpfrac{n-2}{2n(n-1)}} \rfloor^n.
\leqno{\hbox{where}}\]
Thus our result follows from the corresponding result for integers, \fullref{olga3}.
\end{proof}

\subsection[Proof of \ref{fullfill}]{Proof of \fullref{fullfill}} \label{sec33}

Here we show the following.

\begin{theorem}
 \label{fullfill2}
  Suppose that
\begin{gather*}
    b >
    (M_3)^4 a^2
\\  \tag*{\hbox{\rlap{or}}}
    a \ge
    8 \tfrac{1}{36}  \quad\text{and}\quad 
    b >
    (M_3)^2.
\end{gather*}
  Then $E(1,a,b) \hookrightarrow B\big( (ab)^{\unfrac{1}{3}} \big)$.
\end{theorem}

\begin{remark} As $1 \le a \le b$, the hypotheses of \fullref{fullfill2}
are automatically satisfied~if \[
  a^2+b^2 \geq
  \big( 8\tfrac{1}{36} \big)^2
    \biggl( 1 +  ( 8\tfrac{1}{36}  )^2 \biggl(
\frac{\sqrt[4]{3}}{\sqrt[4]{3} - \sqrt[4]{2}} \biggr)^{96} \biggr) >
  1.41 \times 10^{101}.
\vadjust{\goodbreak}\]
\end{remark}

\begin{proof}
First suppose that $b > (M_3)^4 a^2$. We note that this automatically implies
that $b>8\tfrac{1}{36}$ and $\unfrac{\sqrt{b}}{a}>8\tfrac{1}{36}$.
Then we have
\begin{align*}
  E(1,a,b)  \hookrightarrow {}&E( a,\sqrt{b},\sqrt{b} )\\   
&=  a E( 1,\unfrac{\sqrt{b}}{a}, \unfrac{\sqrt{b}}{a} )  \hookrightarrow
  a E( \unfrac{b^{1/4}}{a^{1/2}}, \unfrac{b^{1/4}}{a^{1/2}},
\unfrac{\sqrt{b}}{a} ) \\
&=  a^{\unfrac{1}{2}}b^{\unfrac{1}{4}}E( 1,1,\unfrac{b^{1/4}}{a^{1/2}} )
 \hookrightarrow
  a^{\unfrac{1}{2}}b^{\unfrac{1}{4}} B( \unfrac{b^{1/12}}{a^{1/6}} )\\
  &=
  B((ab)^{1/3}).
\end{align*}
The first embedding is possible since $b>\smash{8\tfrac{1}{36}}$, the second since
$\unfrac{\sqrt{b}}{a} > \smash{8\tfrac{1}{36}}$, and the third by \fullref{olga4}
since $\smash{\unfrac{b^{1/4}}{a^{1/2}}} > M_3$.

Next we assume that $a \geq 8\tfrac{1}{36}$ and $b > (M_3)^2$.
Then we have embeddings
\begin{align*}
  E(1,a,b)  `\hookrightarrow` E(\sqrt{a}, \sqrt{a}, b)   
`=  `\sqrt{a} E( 1,1,\unfrac{b}{\!\sqrt{a}} ) ` \hookrightarrow`
  \sqrt{a} B( \unfrac{b^{1/3}`}{a^{1/6}} ) `=`
  B(`(ab)^{1/3}).
\end{align*}
Now the first embedding relies on $a\ge \smash{8\tfrac{1}{36}}$ and the second exists
by \fullref{olga4} again, since $\unfrac{b}{\sqrt{a}}>\sqrt{b}>M_3$.
\end{proof}

\section{Packing stability in \texorpdfstring{$\CP^n$}{CP\textcircumflex n}}\label{sec4}

In this section we prove \fullref{pack}. It relies on the following
generalization to higher dimensions of a $4$--dimensional polytope decomposition
introduced by McDuff in \cite{mcduff1}.

\begin{lemma}
 \label{weight}
  For any $k \in \N$, \[\tsty\bigsqcup_k B(1) \hookrightarrow E ( 1^{\times (n-1)},k
 ).\] That is, for any $\rho<1$ the disjoint union of $k$ balls of capacity
$\rho$ can be symplectically embedded in the ellipsoid $E(1^{\times( n-1)},k)$.
\end{lemma}

\begin{remark} \label{alternative} Lemma $2.6$ of \cite{mcduff1} proves a
stronger result in dimension $4$, namely that there exists a symplectic
embedding $\bigsqcup_k \r{B}^4(1) \to \r{E} ( 1,k  )$ from the interiors of
the balls to the interior of the ellipsoid. As far as we know, it is unknown if
such an embedding exists in higher dimension. Nevertheless, we give a proof of
\fullref{weight} following this lemma.

There is an alternative approach following Traynor~\cite[Chapters $5$ and $6$]{traynor}; 
see also \cite[Section $9.4$]{Sh1}. The idea is to approximate both the
ellipsoid and the balls by Lagrangian products. Once this is done correctly it
is easy to see that at least an arbitrarily large compact subset of the $k$
open balls can be symplectically embedded in the ellipsoid.
\vadjust{\goodbreak}
\end{remark}

\begin{proof}[Proof of \fullref{weight}]
There is a natural action of the torus $T^n$ on the open ellipsoid $\r{E} (
1^{\times (n-1)},k  ) \subset \CC^n$ given by rotation in each of the $n$
complex coordinates.
Let $\{\v{e}_j\}$ denote the standard basis in $\R^n$. Then the moment polytope
$\Delta $  of the corresponding moment map is the convex hull of the set $\{
0,\v{e}_1, \v{e}_2, \ldots, \v{e}_{n-1},k\v{e}_n \}$ minus the diagonal; that
is, the $n$--dimensional polytope of vectors $p = (x_1, \ldots, x_n)$ that
satisfy
\[x_j \geq 0 \quad\text{and}\quad  x_1 + \cdots + x_{n-1} +\unfrac{x_n}{k} < 1.\]
Inside $\Delta$ we can find $k$ subpolytopes $\Delta_j$ defined by taking the
interior of the convex hull of
$\{ \v{e}_1, \ldots, \v{e}_{n-1}, (j-1)k\v{e}_n, jk\v{e}_n \}$ for $j = 1
\ldots k$; see \fullref{figure3}. The affine map $\Theta\co \R^{n}
\longrightarrow \R^n$ that fixes the first $n-1$ coordinates  and takes the
$n$--th coordinate ~$x_n$ to $x_n - \big( k -(x_1+\cdots+x_{n-1}) \big)$
has integer coefficients and takes each $\Delta_j$ onto the previous polytope
$\Delta_{j-1}$ so that each $\Delta_j$ maps onto $\Delta_1$ by the integral
affine map~ $\Theta^{j-1}$.
    
\begin{figure}[ht!]
\centering
  \includegraphics{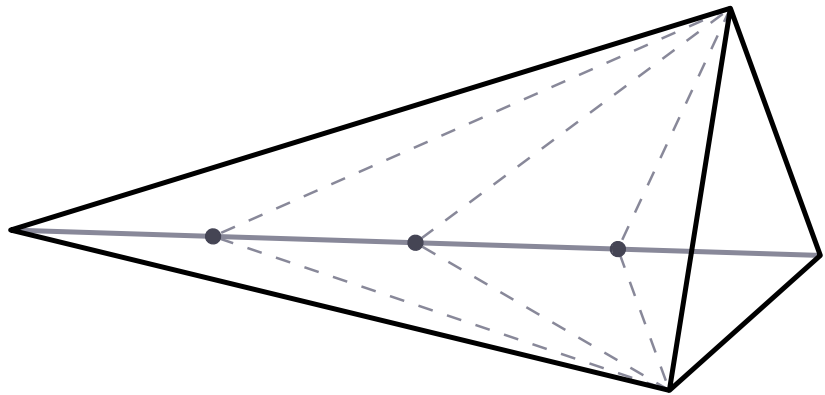}
  \caption{Division of the polytope $\Delta$}
 \label{figure3}
\end{figure}

Now, one can see that $\Delta_1$ is the moment polytope of the open subset
$$U=\{(z_1, \dots, z_n) \in \CC^n\mid z_i \neq 0, \pi|z_1|^2 + \cdots + \pi|z_n|^2
< 1 \} \subset B(1) \subset \CP^n.$$ This contains an embedded symplectic
open ball of capacity $\rho$ for any $\rho<1$, for this, one can apply a
product map such as $\Psi^{\rho}$ defined in \cite[page~420]{traynor}.

From the above, all other $\Delta_j$ must also admit an embedding of a ball of
the same capacity.
Thus, we have found an embedding \[\tsty\bigsqcup_k B(\rho) \to E ( 1^{\times
(n-1)},k  )\] for any $\rho<1$, as required.
\vadjust{\goodbreak}
\end{proof}

\begin{remark} The result above enables us to find embeddings of disjoint balls
into a space whenever ellipsoid embeddings of the specified type are available. 
Theorem ~1.1 in \cite{mcduff1} also gives a converse, that an ellipsoid
embedding exists whenever the corresponding ball embedding does. This result
uses an orbifold blow up as in ~Godinho~ \cite{godh} (see also McDuff 
\cite[Section~2.2]{mcduff5}) then the theory of holomorphic curves and the
rational blow down constructions due to Symington \cite{sym}. No such
technology is yet available in higher dimension.
\end{remark}

\begin{proof}[Proof of \fullref{pack}]
Rescaling, it is required to show that for all natural numbers $k \geq M_n$
there exists an embedding
\[\tsty\bigsqcup_k B(1) \hookrightarrow B^{2n}(k^{1/n}).\]
By \fullref{olga4} there exists an embedding
\[E ( 1^{\times (n-1)},k  ) \hookrightarrow B(k^{1/n}),\]
so it suffices to find an embedding
\[\tsty\bigsqcup_k B(1) \hookrightarrow E ( 1^{\times (n-1)},k  ).\]
But this is \fullref{weight}.
\end{proof}

\begin{remark}[Improving the bounds] \label{improve} 
Our stability bound $M_n$ in \fullref{pack} and \fullref{stability} seem far
from optimal. Improving these bounds will be the topic of a future work
\cite{bh}. This will follow from improving the bound $M_n$ in \fullref{olga4}. Now, according to \cite{mcduff2}, one $4$--dimensional ellipsoid
embeds into another if and only if there are no obstructions coming from the
Embedded Contact Homology capacities defined by M~Hutchings in
\cite{hutchings}. Therefore we have a route to proving \fullref{olga2}
for real rather than just integer $k$ by studying the behavior of the ECH
capacities. Once this is done, \fullref{olga4} follows exactly as \fullref{olga3}.
\end{remark}

\bibliographystyle{gtart}
\bibliography{link}

\end{document}